\documentclass[11pt, leqno]{scrartcl}
\usepackage{amsmath}
\usepackage{amssymb}
\usepackage{amsthm}
\usepackage{mathrsfs}
\usepackage[a4paper, margin=3cm]{geometry}
\usepackage{color}
\usepackage[protrusion=true,expansion=true]{microtype}
\usepackage{enumerate}
\usepackage{numprint}
\usepackage{booktabs}
\usepackage{cprotect}
\usepackage{graphicx}
\usepackage{tikz}
\usetikzlibrary{arrows}
\usetikzlibrary{calc}
\usetikzlibrary{decorations.markings}

\usepackage[colorlinks]{hyperref}

\usepackage[english]{babel}

\newtheorem{thm}{Theorem}
\newtheorem{cor}[thm]{Corollary}

\newtheorem{lem}[thm]{Lemma}
\newtheorem{prop}[thm]{Proposition}
\theoremstyle{definition}
\newtheorem{defn}[thm]{Definition}

\newtheorem{rem}[thm]{Remark}

\newcommand{\CC}{\mathbb{C}}

\newcommand{\eps}{\varepsilon}
\renewcommand{\Im}{\operatorname{Im}}
\renewcommand{\Re}{\operatorname{Re}}
\newcommand{\NN}{\mathbb{N}}
\newcommand{\RR}{\mathbb{R}}

\newcommand{\ZZ}{\mathbb{Z}}
\newcommand{\HH}{\mathbb{H}}

\newcommand{\GL}{\textrm{GL}}
\newcommand{\SL}{\textrm{SL}}

\DeclareOldFontCommand{\bf}{\normalfont\bfseries}{\mathbf}

\title{Analytic evaluation of Hecke eigenvalues for classical modular forms}
\author{David Armend\'ariz\footnote{\url{darmendarizp@estud.usfq.edu.ec}},  Owen Colman\footnote{\url{owencolman@gmail.com}}, Nicol\'as Coloma\footnote{\url{ncoloma@estud.usfq.edu.ec}},\\
Alexandru Ghitza\footnote{\url{aghitza@alum.mit.edu}}, Nathan C.~Ryan\footnote{\url{nathan.ryan@bucknell.edu}}, and Dar\'io Ter\'an\footnote{\url{dteranacaro@ufl.edu}}}
\date{\today}

\begin{document}
\thispagestyle{empty}

\maketitle
\begin{abstract}
  We propose a method for computing approximations to the Hecke eigenvalues of a classical modular eigenform $f$, based on the analytic evaluation of $f$ at points in the upper half plane.
  Our approach works with arbitrary precision, allows for a strict control of the error in the approximation, and outperforms current exact computation methods.


\end{abstract}

\section{Introduction}\label{sect:intro}

Computing Fourier coefficients and Hecke eigenvalues of modular forms has long been of interest.  The relationship between the number of points modulo a prime $p$ on an elliptic curve $E$ and the coefficients of the modular form associated to $E$ is a sample motivation.  It is well-known that for modular forms that are eigenforms for the Hecke operators (and suitably normalized), the Fourier coefficients and Hecke eigenvalues agree.  In this paper we describe a new method to compute Hecke eigenvalues of eigenforms, as well as an implementation of the method.  An analogous version of this algorithm but for Maass forms was first described by Hejhal \cite{hejhal} and optimized by Booker, Str\"ombergsson, and Venkatesh \cite{booker}.

For example, consider the Ramanujan $\tau$-function given by
\[
\sum_{n \geq 1} \tau(n)q^n = q\prod_{n\geq 1} (1-q^n)^{24} = \Delta(z)
\]
where $q=2\pi i z$ and $\Im(z)>0$; $\Delta(z)$ is the unique modular cusp form of level 1 and weight 12.  A 1947 conjecture due to Lehmer \cite{lehmer}, asserts that the Fourier coefficients of $\Delta(z)$ never vanish.  Since Fourier coefficients of $\Delta(z)$ are also its Hecke eigenvalues, it can be said that one of the oldest unsolved conjectures about modular forms can be verified by computing the Hecke eigenvalues of a modular form.  

The standard way to compute Hecke eigenvalues, as is done in Sage \cite{sage} and MAGMA \cite{magma}, is by the method of modular symbols (see \cite{stein} for an algorithmic introduction to modular symbols).  The advantage of this method is that it is quite general; in principle, it can be carried out for modular forms of any weight and level.  The issue is that, as the dimension of the space of modular forms increases (e.g., as the level or weight increases) the linear algebra required for the computation of Hecke eigenvalues becomes quite difficult and inefficient, making it hard to compute the $p$th eigenvalue for moderately large $p$.  Recently, Wuthrich \cite{wuthrich} proposed the computation of modular symbols which is faster for large level compared to traditional methods of computing modular symbols.  Also recently, PARI/GP \cite{pari} has started to include methods to compute Hecke eigenvalues based on trace formulas \cite{belabas}.

A significant breakthrough appeared in the second half of the last decade with the work of Couveignes and Edixhoven and their collaborators~\cite{couveignes}.
Their main result is an algorithm that computes the Galois representation over a finite field attached to a modular eigenform of level one in time polynomial in the logarithm of the cardinality of the finite field.
One can therefore compute the coefficients of such eigenforms (in characteristic zero) via a multimodular algorithm.
Unfortunately, we are not aware of any implementations of this algorithm that are available for public use.

We propose yet another way to compute Hecke eigenvalues.  The idea is relatively simple.  For $f$ a newform, we know that the action of the Hecke operator $T_p$ gives $T_pf = \lambda_pf$, where the equality is of functions.  In other words, if we could evaluate $T_pf$ and $f$ at some fixed point $z_0$ in the complex upper half-plane $\mathbb{H}$ at which $f$ does not vanish, then 
\[
\lambda_p = \frac{T_pf(z_0)}{f(z_0)}.
\]
One clear drawback to our proposed method is that we only get a numerical approximation to $\lambda_p$.  This drawback is not too problematic.  For numerical experiments on $L(f,s)$, the L-function associated to $f$, only numerical approximations to the Dirichlet series coefficients (which are determined by the Hecke eigenvalues of $f$) are required.  If an exact representation of the Hecke eigenvalue is required (e.g., as an element of a number field), since the number field in which the Hecke eigenvalue lives is known, one can use LLL to find the Hecke eigenvalue exactly.  The advantages of our method are that it allows us to compute Hecke eigenvalues more quickly than using the traditional approach and that it can easily be made parallel.
As we indicate below, extensive benchmarks also show that our method exhibits performance that compares very favorably with modular symbol approaches.

In the special case of a level one eigenform $f$, we investigate a variant of the analytic evaluation method that writes $f$ as an explicit polynomial in the Eisenstein series $E_4$ and $E_6$, and evaluates these Eisenstein series at relevant points as described above.

\subsection*{Acknowledgments:}  We thank the Fulbright Commission of Ecuador for making the collaboration of Armend\'ariz, Ter\'an and Ryan possible.  We also thank David Farmer for some helpful conversations.

\section{Modular forms background}

We define the upper half plane $\HH=\{z\in\CC:\Im(z)>0\}$ and the group 
\[
\SL_2(\RR)=\left\{ \begin{pmatrix}a&b\\c&d\end{pmatrix}:a,b,c,d\in\RR\text{ and } ad-bc = 1\right\}.
\]
The group $\SL_2(\RR)$ acts on $\HH$ via fractional linear transformations: if $\gamma=\left(\begin{smallmatrix} a&b\\c&d\end{smallmatrix}\right)$ then $\gamma z=\tfrac{az+b}{cz+d}$.  

We also consider the discrete subgroup $\SL_2(\mathbb{Z})$, as well as the congruence subgroups $\Gamma_0(N)$ (for $N\geq 1$) of $\SL_2(\ZZ)$ given by
\[
  \Gamma_0(N) = \left\{\left(\begin{smallmatrix} a&b\\c&d\end{smallmatrix}\right)\in \SL_2(\ZZ)\colon c\equiv 0\pmod{N}\right\}.
\]

\begin{defn}\label{defn:MF}
A modular form of level $N$ and weight $k$ is a holomorphic function $f:\mathbb{H}\to\CC$ satisfying the following:
\begin{itemize}
\item The function $f$ has the symmetry
\[
f\left(\frac{az+b}{cz+d}\right)=(cz+d)^{k}f(z)
\]
for all elements
$\left(\begin{smallmatrix}a&b\\c&d\end{smallmatrix}\right)$ in
$\Gamma_0(N)$ and all $z\in\mathbb{H}$.
\item The function $f(z)$ is bounded as $z$ approaches the cusps of $\Gamma_0(N)$.
\end{itemize}
The space of all modular forms of weight $k$ and level $N$ is denoted $M_k(N)$.
\end{defn}

\begin{rem}
We define the slash notation:
\[
f\mid_{\gamma,k} (z) = f\left(\frac{az+b}{cz+d}\right)(cz+d)^{-k}\text{ for }\gamma=\left(\begin{smallmatrix}a&b\\c&d\end{smallmatrix}\right)
\]
which allows us to write the functional equation in Definition~\ref{defn:MF} as
\[
f\mid_{\gamma,k} (z)=f(z).
\]
We also observe that for $\gamma\in \GL_2(\RR)$, the group of invertible matrices with entries in $\RR$, we can define
\[
f\mid_{\gamma,k}(z) = \det(\gamma)^{k-1}f\left(\frac{az+b}{cz+d}\right)(cz+d)^{-k}\text{ for }\gamma=\left(\begin{smallmatrix}a&b\\c&d\end{smallmatrix}\right) .
\]
\end{rem}
Since $\left(\begin{smallmatrix} 1& 1\\0&1\end{smallmatrix}\right)\in \Gamma_0(N)$ for all $N\geq 1$, we see that modular forms are periodic and, since they are holomophic on $\HH$ and at the cusps, they admit a Fourier expansion of the form
\[f(q)=\sum_{n\geq 0} a_n(f)q^n.\]

\begin{defn} A cusp form of weight $k$ and level $N$ is a modular form $f$ of weight $k$ and level $N$ that vanishes at the cusps.  In other words, $f$ is a cusp form if its Fourier expansion is of the form $\sum_{n\geq 1}a_n(f)q^n$.  We denote the space of all cusp forms of weight $k$ and level $N$ by $S_k(N)$.
\end{defn}

The arithmetically distinguished cuspforms that we are most interested in are those that are Hecke eigenforms.  Let $f$ be a modular form of weight $k$ and level $N$.  We now define the action of the Hecke operator $T_p$ on $f$ when $p$ and $N$ are relatively prime:
\begin{equation}\label{eq:Tp-a}
  T_p(f)(z) = p^{k-1} f(pz) + \frac{1}{p}\sum_{j\pmod{p}} f\left(\tfrac{z+j}{p}\right).
\end{equation}
Similarly, when $p\mid N$ we get
\begin{equation}\label{eq:Tp-b}
  T_p(f)(z) = \frac{1}{p}\sum_{j\pmod{p}} f\left(\tfrac{z+j}{p}\right).
\end{equation}

Among Hecke eigenforms, we are particularly interested in newforms:
\begin{defn}
If $d_1d_2=N$ and $f\in M_k(d_1)$, then we also have $f\in M_k(N)$ and also $g(z)=f(d_2z)\in M_k(N)$.  The subspace of $S_k(N)$ spanned by the forms obtained in one of these ways are the old forms and the orthogonal complement of the oldforms are the newforms.  The space of newforms is denoted $S_k^{\text{new}}(N)$.
\end{defn}

There is a basis for $S_k^{\text{new}}(N)$ consisting of eigenforms for the Hecke operators $T_p$.  These basis elements, called Hecke eigenforms, are also eigenforms for the Atkin-Lehner operator $W_N$ defined as follows:  let $w_N = \left(\begin{smallmatrix} 0 & -1\\N & 0\end{smallmatrix}\right)$.  Then 
\begin{align*}
W_N(f)(z) &= N^{(2-k)/2} f\mid_{w_N,k}(z)\\ 
&=N^{(2-k)/2} f\left(\frac{-1}{Nz}\right)N^{k-1}(Nz)^{-k}.
\end{align*}
The map $W_N$ is an involution and so $W_N(f)(z)$ is also equal to $\pm f(z)$.  We will use this map to move points away from the real axis.

Two facts about modular forms to which we will refer later on are:
\begin{enumerate}
  \item Deligne's bound on Fourier coefficients: $|a_p(f)| \leq 2p^\frac{k-1}{2}$ and, consequently, $|a_n(f)|\leq d(n) n^{\frac{k-1}{2}}$, where $d(n)$ is the number of divisors of $n$.  We will use this bound when determining how many Fourier coefficients we need in order to be able to evaluate a given modular form $f$ at a given point $z$ in $\HH$. 
  \item The Fourier coefficients of a normalized eigenform are real. This follows from Hecke's theorem relating Hecke eigenvalues to Fourier coefficients, since the Hecke operators are self-adjoint for the Petersson inner product and that the eigenvalues of a self-adjoint operator are real. This fact will be used throughout our computations.
\end{enumerate}

\section{Analytic preliminaries}\label{sec:analytic}

 We begin by stating some results related to bounding the error introduced when we evaluate a given classical modular form and its image under the Hecke operator $T_p$ at a point in the upper half-plane.

\subsection{Error}

We have a quantity defined as
\begin{equation*}
  z=\frac{x}{y}\quad\text{with }x,y\in\CC.
\end{equation*}
The numerator and denominator can be approximated to $x_A$, resp. $y_A$.  In particular, we think of $x_A$ as a numerical approximation to $T_p(f)(z_0)$, for some $z_0\in\HH$, and $y_A$ as $f(z_0)$.

Given $\eps>0$, we seek values of $\eps_x$ and $\eps_y$ ensure that
\begin{equation*}
  \text{if }|x-x_A|<\eps_x\text{ and }|y-y_A|<\eps_y\text{ then }|z-z_A|<\eps?
\end{equation*}

\begin{lem}
  With the above notation, let $e_x=x-x_A$ and $e_y=y-y_A$.
  Then
  \begin{equation*}
    z-z_A=\frac{e_x-e_yz_A}{y_A+e_y}.
  \end{equation*}
\end{lem}
\begin{proof}
 Straightforward calculation
\end{proof}

\begin{prop}
  \label{prop:quot}
  For any $h\in (0,1)$, if
  \begin{equation*}
    \eps_x<\frac{h\eps|y_A|}{2}
    \quad\text{and}\quad
    \eps_y<\min\left\{\frac{(1-h)\eps |y_A|}{2|z_A|}, \frac{|y_A|}{2}\right\},
  \end{equation*}
  then $|z-z_A|<\eps$.
\end{prop}
\begin{proof}
 Under the hypotheses, we have $|y_A+e_y|>|y_A|/2$ so
  \begin{equation*}
    |z-z_A|<\frac{2}{|y_A|}\left(|e_x|+|e_yz_A|\right)
    <h\eps+(1-h)\eps=\eps.
  \end{equation*}
\end{proof}

The value of the parameter $h$ can be chosen in such a way that the calculations of $x_A$ and of $y_A$ are roughly of the same level of difficulty.

In addition to the error in a quotient, we also point out that we need to take into account error in a sum.  In particular, to compute $T_p(f)(z_0)$ for some $z_0\in\HH$ we need to evaluate $f$ at $p+1$ (respectively, $p$) if $p$ does not divide (respectively, does divide) the level $N$.  In particular, suppose $x=\sum_{i=1}^N x^{(i)}$, $x_A = \sum_{i=1}^N x_A^{(i)}$ and that $|x^{(i)}-x^{(i)}_A|<\eps_i$.  We want an $\eps_i$ so that $|x^{(i)}-x^{(i)}_A|<\eps_i$ guarantees $|x-x_A|<\eps$.  By the triangle inequality, we can let $\eps_i = \tfrac{\eps}{N}$.

\subsection{Truncation error for normalised eigenforms}
Let $f\in S_k^\text{new}(\Gamma_0(N))$ be a classical newform of weight $k$ with Fourier expansion
\begin{equation*}
  f(z)=\sum_{n=0}^\infty a_n(f) e^{2\pi i nz}.
\end{equation*}
Given $T\in\NN$, we consider the truncated expansion
\begin{equation*}
  f_T(z)=\sum_{n=0}^T a_n(f) e^{2\pi i nz}.
\end{equation*}
Let
\begin{equation*}
  d=\begin{cases}
    \frac{k+1}{2} & \text{if $k$ is odd},\\
    \frac{k+2}{2} & \text{if $k$ is even}.
  \end{cases}
\end{equation*}

\begin{prop}
  \label{prop:truncate}
  Let $\eps>0$ and $y=\Im(z)$.
  If $T$ is such that
  \begin{equation*}
    T\geq \frac{d}{2\pi y}\qquad\text{and}\qquad
    \frac{d+1}{2\pi y}e^{-2\pi yT}T^d<\eps,
  \end{equation*}
  then $|f(z)-f_T(z)|<\eps$.
\end{prop}
\begin{proof}
  We have
  \begin{equation*}
    |f(z)-f_T(z)|
    =\left|\sum_{n=T+1}^\infty a_n(f) e^{2\pi i nz}\right|
    \leq\sum_{n=T+1}^\infty |a_n(f)| e^{-2\pi y n}.
  \end{equation*}
  Using Deligne's bound we have
  \begin{equation*}
    |a_n(f)|\leq d(n) n^{(k-1)/2}\leq n^{(k+1)/2}\leq n^d,
  \end{equation*}
  which gives us
  \begin{equation*}
    |f(z)-f_T(z)|\leq \sum_{n=T+1}^\infty n^d e^{-2\pi yn}.
  \end{equation*}
  Letting $g(x)=x^de^{-2\pi yx}$ we find easily that $g$ is decreasing for $x\geq d/(2\pi y)$.
  By hypothesis we have $T\geq d/(2\pi y)$, so we can compare the above infinite series with the improper integral
  \begin{equation*}
    |f(z)-f_T(z)|\leq \int_T^\infty x^de^{-2\pi yx} dx
    =\frac{e^{-2\pi yT}}{2\pi y}\sum_{j=0}^d\frac{d!}{j!(2\pi y)^{d-j}}T^j.
  \end{equation*}
  For any $j$ between $0$ and $d$, since $T\geq d/(2\pi y)$ we have
  \begin{equation*}
    \frac{(j+1)\cdots d}{(2\pi y)^{d-j}}\leq \left(\frac{d}{2\pi y}\right)^{d-j}<T^{d-j},
  \end{equation*}
  so
  \begin{equation*}
    |f(z)-f_T(z)|\leq \frac{e^{-2\pi yT}}{2\pi y}(d+1)T^d,
  \end{equation*}
  which is smaller than $\eps$ by hypothesis.
\end{proof}

The conditions on $T$ make it clear that larger values of $y$ allow for shorter truncations.

When our modular form is of level 1, we can use the modularity of $f$ to move the evaluation from the point $z=x+iy$ to another point $(az+b)/(cz+d)$ that has a larger imaginary part:
\begin{equation*}
  \Im\left(\frac{az+b}{cz+d}\right)=\frac{y}{(cx+d)^2+c^2y^2}.
\end{equation*}
The group $\SL_2(\ZZ)$ is generated by the matrices
\begin{equation*}
T = \begin{bmatrix} 1&1\\0&1 \end{bmatrix}
  \qquad\text{and}\qquad
S=\begin{bmatrix} 0 & 1\\-1&0 \end{bmatrix}.
\end{equation*}
The former does not change the imaginary part, so it is enough to consider the latter:
\begin{equation*}
  z=x+iy\longmapsto -\frac{1}{z}=-\frac{x}{x^2+y^2}+\frac{y}{x^2+y^2}\,i.
\end{equation*}
This leads us to a very simple approach: given $z=x+iy$, first use the periodicity of $f$ to ensure $|x|\leq 1/2$.
If $x^2+y^2\geq 1$, then evaluate $f$ at $z$; otherwise evaluate $f$ at $-1/z$ and use
\begin{equation*}
  f(z)=z^{-k}f\left(-\frac{1}{z}\right).
\end{equation*}

\begin{rem}\label{rem:inversion}
  In the case of modular forms of level $N>1$, it is harder to guarantee a large imaginary part.
  For example, consider $\Gamma_0(2)$.
  Let $\mathcal{F}$ be a fundamental domain for $\SL_2(\ZZ)$
\[
\mathcal{F}=\left\{z\in\HH: |\Re(z)| \leq \tfrac12 \text{ and } |z|\geq 1\right\},
\]
and $S$ and $T$ be the matrices defined above.  Then a fundamental domain for $\Gamma_0(2)$ is $\mathcal{F}\cup ST\mathcal{F}\cup S\mathcal{F}$.  Since $S\mathcal{F}$ extends down towards the origin, there is no obvious way to move a purely imaginary number with a small imaginary part (say less than 1) to a point with a larger imaginary part (say greater than 1).  

We show how to overcome this difficulty in the cases $N=2,3$ by using Atkin-Lehner operators.
For $N=1$, the matrix $\left(\begin{smallmatrix}0 & -1\\1&0\end{smallmatrix}\right)$ sends a point inside the unit circle to one outside the unit circle.
This inversion along with translations (recall $f(z+m)=f(z)$ for all $m\in \ZZ$) allows us to assume an imaginary part bigger than $1$.
The matrix $w_N = \left(\begin{smallmatrix}0&-1\\N&0\end{smallmatrix}\right)$ is inversion across the circle centered at the origin and of radius $1/\sqrt{N}$.
Figure~\ref{fig:atkin} shows the outcome of applying this inversion for the cases $N=1,2,3,4$ in the computation of the Hecke operator $T_5$.
We can see that the imaginary part increases after applying $w_N$, but that for $N=4$ (and higher) this increase becomes too negligeable to be useful.
Related techniques (such as using $q$-expansions with respect to cups other than $i\infty$) may be appropriate for higher levels.
\end{rem}

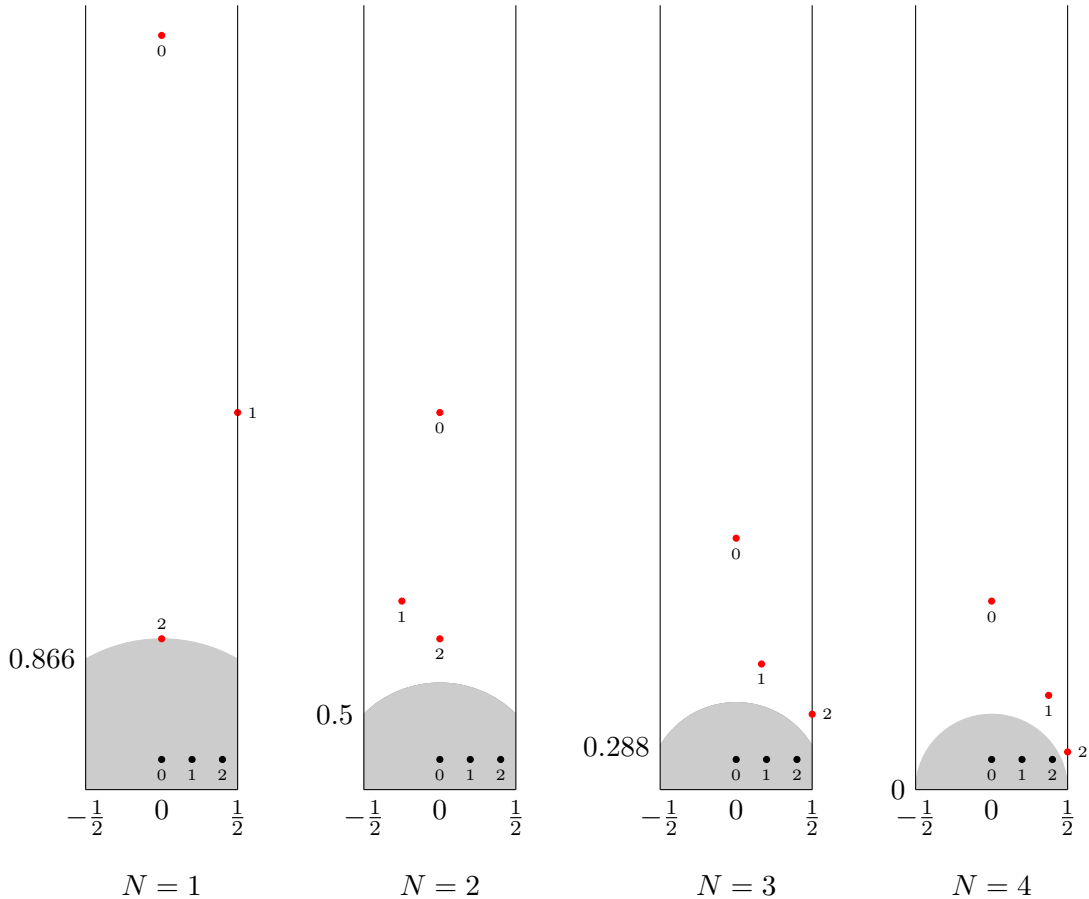
\begin{figure}[h]
  \label{fig:atkin}
\begin{tikzpicture}[scale=2]
  \coordinate[label=below:$0$] (O) at (0,0);
  \coordinate[label=below:{$N=1$}] (lbl) at (0,-0.5);
  \coordinate[label=below:$-\frac{1}{2}$] (left) at (-0.5, 0);
  \coordinate[label=below:$\frac{1}{2}$] (right) at (0.5, 0);
  \coordinate[label=left:$0.866$] (bd) at (-0.5,0.866);
  \fill[black!20] (-0.5,0.866) -- (0.5,0.866) -- (0.5,0) -- (-0.5,0) -- cycle;
  \draw[fill,black!20] ($(O) + (60:1)$) arc (60:120:1);
  \coordinate[label=below:{\tiny $0$}] (h0) at (0,0.2);
  \coordinate[label=below:{\tiny $1$}] (h1) at (0.2,0.2);
  \coordinate[label=below:{\tiny $2$}] (h2) at (0.4,0.2);
  \draw[fill] (h0) circle (0.02);
  \draw[fill] (h1) circle (0.02);
  \draw[fill] (h2) circle (0.02);
  \coordinate[label=below:{\tiny $0$}] (i0) at (0,5);
  \coordinate[label=right:{\tiny $1$}] (i1) at (0.5,2.5);
  \coordinate[label=above:{\tiny $2$}] (i2) at (0,1);
  \draw[fill,red] (i0) circle (0.02);
  \draw[fill,red] (i1) circle (0.02);
  \draw[fill,red] (i2) circle (0.02);
  \draw (-0.5,0) -- (0.5,0);
  \draw (left) -- (-0.5,5.2);
  \draw (right) -- (0.5,5.2);
\end{tikzpicture}
\quad
\begin{tikzpicture}[scale=2]
  \coordinate[label=below:$0$] (O) at (0,0);
  \coordinate[label=below:{$N=2$}] (lbl) at (0,-0.5);
  \coordinate[label=below:$-\frac{1}{2}$] (left) at (-0.5, 0);
  \coordinate[label=below:$\frac{1}{2}$] (right) at (0.5, 0);
  \coordinate[label=left:$0.5$] (bd) at (-0.5,0.5);
  \draw[gray] ($(O) + (45:0.7071)$) arc (45:135:0.7071);
  \fill[black!20] (-0.5,0.5) -- (0.5,0.5) -- (0.5,0) -- (-0.5,0) -- cycle;
  \draw[fill,black!20] ($(O) + (45:0.7071)$) arc (45:135:0.7071);
  \draw (-0.5,0) -- (0.5,0);
  \draw (left) -- (-0.5,5.2);
  \draw (right) -- (0.5,5.2);
  \coordinate[label=below:{\tiny $0$}] (h0) at (0,0.2);
  \coordinate[label=below:{\tiny $1$}] (h1) at (0.2,0.2);
  \coordinate[label=below:{\tiny $2$}] (h2) at (0.4,0.2);
  \draw[fill] (h0) circle (0.02);
  \draw[fill] (h1) circle (0.02);
  \draw[fill] (h2) circle (0.02);
  \coordinate[label=below:{\tiny $0$}] (i0) at (0,2.5);
  \coordinate[label=below:{\tiny $1$}] (i1) at (-0.25,1.25);
  \coordinate[label=below:{\tiny $2$}] (i2) at (0,1);
  \draw[fill,red] (i0) circle (0.02);
  \draw[fill,red] (i1) circle (0.02);
  \draw[fill,red] (i2) circle (0.02);
\end{tikzpicture}
\quad
\begin{tikzpicture}[scale=2]
  \coordinate[label=below:$0$] (O) at (0,0);
  \coordinate[label=below:{$N=3$}] (lbl) at (0,-0.5);
  \coordinate[label=below:$-\frac{1}{2}$] (left) at (-0.5, 0);
  \coordinate[label=below:$\frac{1}{2}$] (right) at (0.5, 0);
  \coordinate[label=left:$0.288$] (bd) at (-0.5,0.288);
  \draw[gray] ($(O) + (30:0.57735)$) arc (30:150:0.57735);
  \fill[black!20] (-0.5,0.288) -- (0.5,0.288) -- (0.5,0) -- (-0.5,0) -- cycle;
  \draw[fill,black!20] ($(O) + (30:0.57735)$) arc (30:150:0.57735);
  \draw (-0.5,0) -- (0.5,0);
  \draw (left) -- (-0.5,5.2);
  \draw (right) -- (0.5,5.2);
  \coordinate[label=below:{\tiny $0$}] (h0) at (0,0.2);
  \coordinate[label=below:{\tiny $1$}] (h1) at (0.2,0.2);
  \coordinate[label=below:{\tiny $2$}] (h2) at (0.4,0.2);
  \draw[fill] (h0) circle (0.02);
  \draw[fill] (h1) circle (0.02);
  \draw[fill] (h2) circle (0.02);
  \coordinate[label=below:{\tiny $0$}] (i0) at (0,1.667);
  \coordinate[label=below:{\tiny $1$}] (i1) at (0.167,0.833);
  \coordinate[label=right:{\tiny $2$}] (i2) at (0.5,0.5);
  \draw[fill,red] (i0) circle (0.02);
  \draw[fill,red] (i1) circle (0.02);
  \draw[fill,red] (i2) circle (0.02);
\end{tikzpicture}
\quad
\begin{tikzpicture}[scale=2]
  \coordinate[label=below:$0$] (O) at (0,0);
  \coordinate[label=below:{$N=4$}] (lbl) at (0,-0.5);
  \coordinate[label=below:$-\frac{1}{2}$] (left) at (-0.5, 0);
  \coordinate[label=below:$\frac{1}{2}$] (right) at (0.5, 0);
  \coordinate[label=left:$0$] (bd) at (-0.5,0);
  \draw[fill,black!20] ($(O) + (0:0.5)$) arc (0:180:0.5);
  \draw (-0.5,0) -- (0.5,0);
  \draw (left) -- (-0.5,5.2);
  \draw (right) -- (0.5,5.2);
  \coordinate[label=below:{\tiny $0$}] (h0) at (0,0.2);
  \coordinate[label=below:{\tiny $1$}] (h1) at (0.2,0.2);
  \coordinate[label=below:{\tiny $2$}] (h2) at (0.4,0.2);
  \draw[fill] (h0) circle (0.02);
  \draw[fill] (h1) circle (0.02);
  \draw[fill] (h2) circle (0.02);
  \coordinate[label=below:{\tiny $0$}] (i0) at (0,1.25);
  \coordinate[label=below:{\tiny $1$}] (i1) at (0.375,0.625);
  \coordinate[label=right:{\tiny $2$}] (i2) at (0.5,0.25);
  \draw[fill,red] (i0) circle (0.02);
  \draw[fill,red] (i1) circle (0.02);
  \draw[fill,red] (i2) circle (0.02);
\end{tikzpicture}
\caption{The Atkin-Lehner operator $W_N$ performs inversion in the circle of radius $1/\sqrt{N}$.  These circles are illustrated above for $N=1,\dots,4$.  We see that the minimal guaranteed imaginary part decreases from $\sqrt{3}/2\approx 0.866$ in level $1$ to $0$ in level $4$ (and higher). We marked the points at which we need to evaluate $f$ in order to estimate $T_5(f)$ at $z=i$, as well as the corresponding points after optimizing the imaginary part via the Atkin-Lehner operator.}
\end{figure}

\section{Implementation details}

We now describe some of the details of our implementation, which can be found at~\cite{code}.
All inexact arithmetic, as well as elementary function evaluations, are done using ball arithmetic via Fredrik Johansson's library arb~\cite{arb}, as packaged in Sage.
This allows us to control the propagation of errors at the level of simple operations.

\subsection{Guaranteed error bounds}

In Section~\ref{sec:analytic} we stated and proved some error bounds.  Recall the notation of Proposition~\ref{prop:quot}.  In particular, we approximate $\lambda_p$ ($z_A$ in the notation of the Proposition) by taking the quotient $\tfrac{T_p(f)(z_0)}{z_0}$ (denoted $\tfrac{x_A}{y_A}$ in the Proposition).  We want to compute $\lambda_p$ to a chosen number of digits of accuracy.

In order to use the results of Proposition~\ref{prop:quot} in practice, we need a lower bound on $|y_A|$  and an upper bound on $|z_A|$ (which can be obtained from the lower bound on $|y_A|$ and an upper bound on $|x_A|$).

How do we bound $|x_A|$?  We compute a very coarse estimate $\tilde{x}$ to $x$, with $\tilde{\eps}_x$ just small enough that $|\tilde{x}|-2\tilde{\eps}_x>0$.  (We can start with $\tilde{\eps}_x=0.1$ and keep dividing by $10$ until the condition holds.)  Later we will make sure that $\eps_x$ is smaller than $\tilde{\eps}_x$.

Then we know that
\begin{equation*}
  |\tilde{x}-x|<\tilde{\eps}_x\qquad\text{and}\qquad
  |x_A-x|<\eps_x\leq\tilde{\eps}_x,
\end{equation*}
so
\begin{equation*}
\big||x_A|-|\tilde{x}|\big|\leq |x_A-\tilde{x}|<2\tilde{\eps}_x\qquad\Rightarrow\qquad 0< |\tilde{x}|-2\tilde{\eps}_x < |x_A| < |\tilde{x}|+2\tilde{\eps}_x,
\end{equation*}
giving us lower and upper bounds on $|x_A|$.  A similar argument works for $|y_A|$.

\subsection{Finding where to truncate}

Our code finds the $T$ described in Proposition~\ref{prop:truncate} as follows.   We increment $T$ by 1 until both conditions of Proposition~\ref{prop:truncate} are met.  

The user is free to choose any value of $z_0$ at which to evaluate the modular form; when $z_0$ is chosen to have a moderately large imaginary part, then $T=1$ satisfies both conditions.  This is problematic because then the difference $|f(z)-f_T(z)|$ becomes zero and we are not really approximating anything. In this case, we arbitrarily increase $T$ to 100.

\subsection{Optimizing the choice of $z_0$}

Our code allows for the user to choose any value of $z_0$ at which to evaluate the modular form $f$.  When $z_0$ is closer to the real axis, our algorithm slows down considerably.  For such $z_0$ we need to compute a very large number of Fourier coefficients since in Proposition~\ref{prop:truncate} the value of $T$ becomes quite large.  

As described in Remark~\ref{rem:inversion}, we use inversion across the unit circle or the circle centered at the origin and of radius $1/\sqrt{N}$ to make the imaginary part bigger.


\subsection{A mirror relation and the case for the imaginary axis}
\begin{prop}
  Let $N\in\NN$ and let $\{a_n\}$ be a collection of real numbers indexed by integers $n\in\ZZ$.
  Let
  \begin{equation*}
    f(z)=\sum_{n=-\infty}^{\infty} a_n e^{2\pi i nz/N},\qquad
    \text{for }z\in\HH.
  \end{equation*}
  Then
  \begin{equation*}
    f(-\overline{z})=\overline{f(z)}\qquad
    \text{for all }z\in\HH.
  \end{equation*}
\end{prop}

\begin{rem}  We point out that, even though the sum in Proposition goes from $-\infty$ to $+\infty$, since we are considering Hecke eigenforms, the sums we consider below go from $1$ to $+\infty$.
\end{rem}

\begin{proof}
  Write $z=x+iy\in\HH$, then $-\overline{z}=-x+iy\in\HH$.
  We have
  \begin{align*}
    e^{2\pi i nz/N} &= e^{-2\pi ny/N}
    \left(\cos(2\pi nx/N)+i\sin(2\pi nx/N)\right)\\
    e^{-2\pi i n\overline{z}/N} &= e^{-2\pi ny/N}
    \left(\cos(2\pi nx/N)-i\sin(2\pi nx/N)\right)
    =\overline{e^{2\pi i nz/N}},
  \end{align*}
  from which the claim follows because the coefficients $a_n$ are in $\RR$.
\end{proof}

\begin{cor}
  Given $f$ as in the Proposition:
  \begin{enumerate}[(a)]
    \item $f(iy)\in\RR$ for all $y\in\RR_{>0}$;
    \item if $y\in\RR_{>0}$, $p$ is prime and $b\in\{1,\dots,p-1\}$ then
      \begin{equation*}
        f\left(\frac{iy-b}{p}\right)=
        \overline{f\left(\frac{iy+b}{p}\right)}.
      \end{equation*}

\item Let $f\in M_k(\SL_2(\ZZ))$.  The summands in $T_pf(iy)$ (for $p>2$) come in pairs of conjugate complex numbers, which allows us to reduce the necessary computation in half.  In particular,
\begin{equation*}
  (T_p f)(iy) = p^{k-1} f(iyp) +\frac{1}{p} f\left(i\frac{y}{p}\right)
  + \frac{2}{p}\Bigg(\Re f\left(\frac{iy+1}{p}\right)+\dots+
  \Re f\left(\frac{iy+\frac{p-1}{2}}{p}\right)\Bigg).
\end{equation*}
  \end{enumerate}
\end{cor}
\begin{proof}
  Apply the Proposition with (a) $z=iy$ and (b) $z=(iy+b)/p$.  For (c) we apply parts (a) and (b) and use the fact that $f(z+1)=f(z)$.
\end{proof}

\begin{rem}
When $p=2$ there are no savings to be gained by using this symmetry.  Also, the corollary allows us to carry out computations over the real numbers.
\end{rem}


\subsection{Comparison of the two methods}

In this section, we compare the modular symbol method as implemented in Sage \cite{sage} to our method.  In Figure~\ref{fig:err} we see the results of the following computation:  for each $p$ we compute the $p$th Hecke eigenvalue in both ways (using modular symbols and using our analytic approach) and take $\log_{10}$ of the absolute value of the difference.  We see in Figure~\ref{fig:err} that for a fixed $F\in S_{24}(1)$ the two methods agree to at least $10^{-8}$.  In this section we use $\lambda_p$ to denote the $p$th Hecke eigenvalue computed via modular symbols and $\tilde{\lambda}_p$ to denote the $p$th Hecke eigenvalue approximated using our analytic approach.

\begin{figure}
\begin{center}
\includegraphics[width=3in]{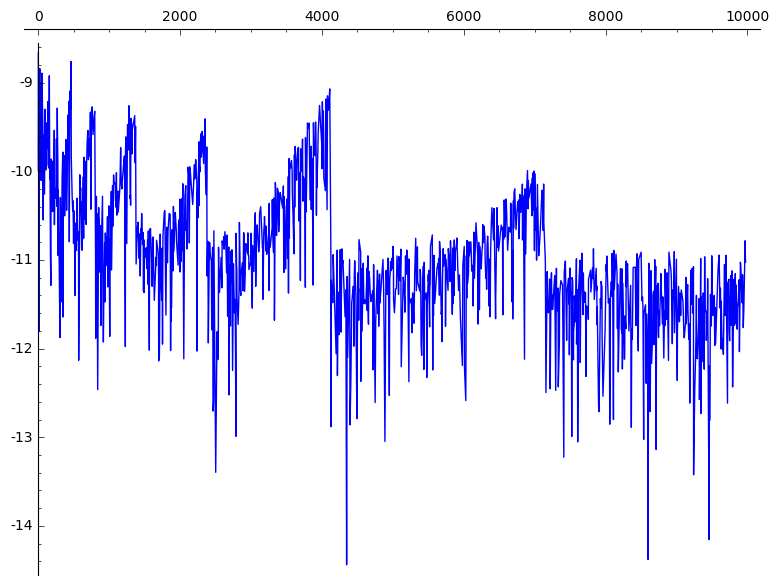}
\end{center}
\caption{Let $\lambda_p$ be the $p$th Hecke eigenvalue computed using modular symbols and $\tilde{\lambda}_p$ be the $p$th Hecke eigenvalue approximated using our analytic approach.  The plot here is $\log_{10}(|\lambda_p-\tilde{\lambda}_p|)$.}\label{fig:err}
\end{figure}

In Table~\ref{tbl:classical-timings} we highlight some timings in which we computed Hecke eigenvalues using both methods.  The results of our method are computed via code like
\begin{verbatim}
sage: F12 = Newforms(1, 12, names='a')[0]
sage: p = next_prime(10000)
sage: err = 0.1
sage: e = eigenvalue_numerical(F12, p, err)
sage: e
-5.758585642481476962744000?e21
\end{verbatim}
and the results of the modular symbols method are computed via code like\footnote{We have chosen this somewhat cryptic way of computing the eigenvalue via modular symbols over more obvious and readable ones for two reasons: (a) Sage uses aggressive caching of various intermediate results, which makes it hard to obtain useful timings from repeated evaluations of the same eigenvalue; (b) it is the quickest way of getting the eigenvalue, since the more user-facing methods perform other bookkeeping that is not absolutely necessary for our purpose and can easily take up the bulk of the running time.  For instance, \texttt{e = F12[p]} for the example given in the text takes a little over 9 minutes, compared to about 1 second for the more verbose modular symbols method.}
\begin{verbatim}
sage: F12 = Newforms(1, 12, names='a')[0]
sage: p = next_prime(10000)
sage: ms = F12.modular_symbols(1)
sage: T = ms._eigen_nonzero_element(p)
sage: e = ms._element_eigenvalue(T, name='b')
sage: e
-5758585642481476962744
\end{verbatim}

  \begin{table}[h]
    \centering
    \begin{tabular}{rrr|rrr}
      \toprule
      level & weight & $p$ & modular symbols & analytic & Eisenstein\\
      \midrule
      $1$ & $12$ & $\numprint{   1009}$ &   $0.073$ &   $\mathbf{0.051}$ & $0.071$\\
        && $\numprint{  10007}$ &   $0.967$ &   $\mathbf{0.557}$ & $0.847$\\
        && $\numprint{ 100003}$ &  $13.710$ &  $\mathbf{7.270}$ & $9.274$\\
        && $\numprint{1000003}$ & $163.515$ & $\mathbf{79.494}$ & $100.782$\\
      \midrule
    $1$ & $24$ & $\numprint{   1009}$ &   $0.158$ &   $\mathbf{0.069}$ & $0.085$\\
        && $\numprint{  10007}$ &   $2.023$ &   $\mathbf{0.751}$ & $0.935$\\
        && $\numprint{ 100003}$ &  $27.190$ &  $\mathbf{8.742}$ & $9.941$\\
        && $\numprint{1000003}$ & $330.854$ & $\mathbf{95.078}$ & $108.923$\\
      \midrule
    $1$ & $100$ & $\numprint{101}$ & $0.057$ & $0.043$ & $\mathbf{0.013}$\\
        && $\numprint{   1009}$ &   $0.918$ &   $0.233$ & $\mathbf{0.146}$\\
        && $\numprint{  10007}$ &   $14.197$ &   $1.938$ & $\mathbf{1.913}$\\
        && $\numprint{ 100003}$ &   $186.873$ &  $24.117$ & $\mathbf{23.449}$\\
      \midrule
    $1$ & $200$ & $\numprint{101}$ &  $0.171$ & $1.160$ & $\mathbf{0.024}$\\
        && $\numprint{   1009}$ &   $2.443$ &   $3.101$ & $\mathbf{0.347}$\\
        && $\numprint{  10007}$ &   $36.195$ &   $9.965$ & $\mathbf{5.104}$\\
        && $\numprint{ 100003}$ &   $500.388$ &  $84.217$ &  $\mathbf{69.916}$\\
      \midrule
      $2$ & $8$ & $\numprint{   1009}$ &   $\mathbf{0.050}$ &   $0.055$\\
          && $\numprint{  10007}$ &   $0.681$ &   $\mathbf{0.596}$\\
          && $\numprint{ 100003}$ &   $8.452$ &   $\mathbf{6.627}$\\
          && $\numprint{1000003}$ & $108.695$ & $\mathbf{75.283}$\\
      \midrule
      $2$ & $48$ & $\numprint{   1009}$ &   $0.360$ &   $\mathbf{0.108}$\\
          && $\numprint{  10007}$ &   $4.927$ &   $\mathbf{1.326}$\\
          && $\numprint{ 100003}$ &   $67.787$ &   $\mathbf{15.784}$\\
      \midrule
      $3$ & $6$ & $\numprint{   1009}$ &   $\mathbf{0.041}$ &   $0.055$\\
          && $\numprint{  10007}$ &   $\mathbf{0.542}$ &   $0.615$\\
          && $\numprint{ 100003}$ &   $\mathbf{6.809}$ &   $6.881$\\
          && $\numprint{1000003}$ & $87.462$ & $\mathbf{78.406}$\\
      \midrule
      $3$ & $24$ & $\numprint{   1009}$ &   $0.206$ &   $\mathbf{0.097}$\\
          && $\numprint{  10007}$ &   $2.642$ &   $\mathbf{1.070}$\\
          && $\numprint{ 100003}$ &   $34.894$ &   $\mathbf{12.209}$\\
      \bottomrule
    \end{tabular}
    \caption{A summary of timings (in seconds; the \textbf{bold} figures identify the shortest time in each row) to compute Hecke eigenvalues in different ways.  The first uses the modular symbols method as implemented in Sage; the second way uses analytic evaluation of the eigenform itself; the third way, valid only in level $1$, uses analytic evaluation of the Eisenstein series $E_4$ and $E_6$.  The computations were done on one Intel i7-8550U core at 1.80GHz, on a machine with 16 GB RAM.  The benchmarks were repeated several times using the \texttt{pytest-benchmark} framework and the median timing is reported here.}\label{tbl:classical-timings} 
  \end{table}


A clear difference between the two methods is the number of Fourier coefficients required for the computation of the $p$th Hecke eigenvalue.  In Figure~\ref{fig:T} we see how few coefficients are needed to compute $\tilde{\lambda}_p$.

\begin{figure}
\begin{center}
\includegraphics[width=3in]{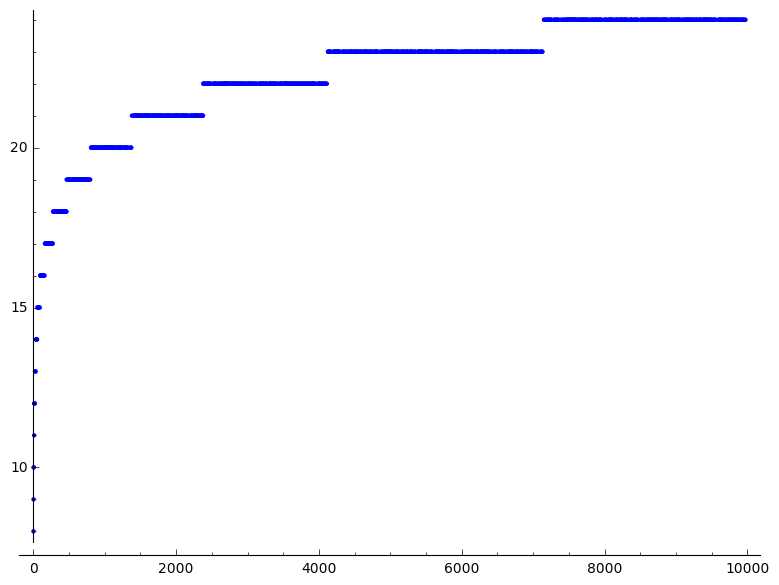}
\end{center}
\caption{For each $p$ we calculate the number of coefficients of $F\in S_{24}(1)$ to compute $\tilde{\lambda}_p$, our approximation to the $p$th Hecke eigenvalue $\lambda_p$ using our analytic approach.  In principle, to compute $\lambda_p$ using the Fourier expansion of the modular form, one would need the first $p$ coefficients of the modular form.}\label{fig:T}
\end{figure} 

\subsection{Improvements in Level 1} In the previous sections we described how to evaluate $f$ using its Fourier expansions and in Table~\ref{tbl:classical-timings} we provide some timings using this method.  In this section we focus on the specific case of level 1 and describe a faster way to evaluate $f$.  It is well known that the ring of modular forms of level 1 is generated by the Eisenstein series of levels 4 and 6.  What we do in this case is the following:  Let $z\in\HH$:
\begin{enumerate}
\item\label{step1} calculate roughly $k/12$ coefficients of the modular form $f$, $E_4$ and $E_6$;
\item\label{step2} identify a basis for $S_k(1)$ by finding $T = \{(a,b)\in\ZZ:4a+6b=k-12 \}$ and computing the first $k/12$ coefficients of $\Delta\cdot E_4^a E_6^b$ for all $(a,b)\in T$;
\item express $f$ as a linear combination of the basis elements found in Step \ref{step2} and replace $\Delta$ with $\tfrac{E_6^2-E_4^3}{1728}$;
\item\label{step4} calculate $x=E_4(z)$, $y=E_6(z)$ and the evaluate $f(z)$ by taking the appropriate algebraic combination of $x$ and $y$.
\end{enumerate}
\begin{rem}  Some comments on the steps are in order: 
\begin{itemize}
\item In Step~\ref{step1}, the number of coefficients needed for large weight is much smaller than the number of coefficients needed to evaluate $f$ using its Forier expansion.
\item In Step~\ref{step4}, we use Fredrik Johansson's arb package \cite{arb}, available in Sage \cite{sage} to evaluate $E_4$ and $E_6$.
\end{itemize}
\end{rem}

Numerical experiments suggest that this approach via evaluation of Eisenstein series is particularly beneficial in the case of large weights (see Table~\ref{tbl:classical-timings}).


\bibliographystyle{plain}
\bibliography{hecke}

\end{document}